\documentclass[a4paper,10pt,reqno]{amsart}
\usepackage[utf8x]{inputenc}

\usepackage{amsmath,amssymb,amsthm,amsxtra,mathrsfs,cancel,bm,paralist}
\usepackage[longnamesfirst,numbers]{natbib}
\usepackage{url}
\usepackage[colorlinks,citecolor=blue]{hyperref}

\newcommand{\N}{\mathbb{N}}

\newcommand{\R}{\mathbb{R}}

\newcommand{\uu}{\boldsymbol{u}}
\newcommand{\UU}{\boldsymbol{U}}
\newcommand{\vv}{\boldsymbol{v}}
\newcommand{\ww}{\boldsymbol{w}}
\newcommand{\zz}{\boldsymbol{z}}

\newcommand{\bom}{\boldsymbol{\omega}}
\newcommand{\BA}{\boldsymbol{A}}
\newcommand{\BB}{\boldsymbol{B}}

\newcommand{\ff}{\boldsymbol{f}}

\newcommand{\Jeps}{\mathcal{J}_{\eps}}

\newcommand{\eps}{\varepsilon}
\newcommand{\weakto}{\rightharpoonup}
\newcommand{\weakstarto}{\overset{*}{\rightharpoonup}}

\newcommand{\BMO}{\mathrm{BMO}}

\newcommand{\pd}{\partial}
\newcommand{\rd}{\mathrm{d}}
\newcommand{\Grad}{\nabla}
\newcommand{\Div}{\nabla \cdot}
\newcommand{\Curl}{\nabla \times}
\newcommand{\Laplace}{\Delta}

\DeclareMathOperator*{\esssup}{ess\,sup}

\newcommand{\abs}[1]{\left| #1 \right|}
\newcommand{\norm}[1]{\| #1 \|}
\newcommand{\bignorm}[1]{\left\| #1 \right\|}
\newcommand{\inner}[2]{\langle #1 , #2 \rangle}

\theoremstyle{plain}
\newtheorem{theorem}{Theorem}[section]
\newtheorem{proposition}[theorem]{Proposition}
\newtheorem{lemma}[theorem]{Lemma}

\newtheorem{corollary}[theorem]{Corollary}

\newtheorem{definition}[theorem]{Definition}

\numberwithin{equation}{section}

\title[Existence and uniqueness for a coupled parabolic-elliptic model]{Existence and uniqueness for a coupled parabolic-elliptic model with applications to magnetic relaxation}

\author[D.\ S.\ McCormick]{David S.\ McCormick}
\thanks{DSMcC is a member of the Warwick ``MASDOC'' doctoral training centre, which is funded by EPSRC grant EP/HO23364/1. JCR is supported by an EPSRC Leadership Fellowship EP/G007470/1.}
\address{D.\ S.\ McCormick \\
Mathematics Institute \\
University of Warwick \\
Coventry, CV4 7AL \\
United Kingdom}
\email{d.s.mccormick@warwick.ac.uk}

\author[J.\ C.\ Robinson]{James C.\ Robinson}
\address{J.\ C.\ Robinson \\
Mathematics Institute \\
University of Warwick \\
Coventry, CV4 7AL \\
United Kingdom}
\email{j.c.robinson@warwick.ac.uk}

\author[J.\ L.\ Rodrigo]{Jose L.\ Rodrigo}
\address{J.\ L.\ Rodrigo \\
Mathematics Institute \\
University of Warwick \\
Coventry, CV4 7AL \\
United Kingdom}
\email{j.l.rodrigo@warwick.ac.uk}

\date{\today}

\keywords{Magnetic relaxation, coupled parabolic-elliptic problem, Ladyzhenskaya inequality, interpolation inequality.}

\subjclass[2010]{
	Primary: 35M33, 35Q35, 76W05. Secondary: 46B70, 46E30.
}

\makeatletter
\hypersetup{
	pdftitle = {\@title},
	pdfauthor = {\shortauthors},
	pdfsubject = {2010 MSC: \@subjclass},
	pdfkeywords = {\@keywords},
}
\makeatother

\begin{document}

\begin{abstract}
We prove existence, uniqueness and regularity of weak solutions of a coupled parabolic-elliptic model in two dimensions; we consider the standard equations of magnetohydrodynamics with the advective terms removed from the velocity equation. Despite the apparent simplicity of the model, the proof requires results that are at the limit of what is available, including elliptic regularity in $L^{1}$ and a strengthened form of the Ladyzhenskaya inequality
\[
\norm{f}_{L^{4}} \leq c \norm{f}_{L^{2,\infty}}^{1/2} \norm{\Grad f}_{L^{2}}^{1/2},
\]
which we derive using the theory of interpolation. The model has applications to the method of magnetic relaxation, introduced by Moffatt (J.~Fluid.~Mech.\ \textbf{159}, 359--378, 1985), to construct stationary Euler flows with non-trivial topology.
\end{abstract}

\maketitle

\section{Introduction}

In this paper we prove global existence and uniqueness of solutions to the following coupled parabolic-elliptic system of equations related to magnetohydrodynamics (MHD), for a velocity field $\uu$, a magnetic field $\BB$ and a pressure field $p$ defined on $\Omega \subset \R^{2}$, as follows:
\begin{subequations}
\label{eqn:StokesMHD}
\begin{align}
- \nu \Laplace \uu + \Grad p_{*} &= (\BB \cdot \Grad) \BB, \label{eqn:StokesMHD-u} \\
\frac{\pd \BB}{\pd t} + (\uu \cdot \Grad) \BB - \eta \Laplace \BB &= (\BB \cdot \Grad) \uu, \label{eqn:StokesMHD-B} \\
\Div \uu = \Div \BB &= 0, \label{eqn:StokesMHD-div}
\end{align}
\end{subequations}
where $p_{*} = p + \frac{1}{2} \abs{\BB}^{2}$ is the total pressure. Here $\nu > 0$ is the coefficient of viscosity, and $\eta > 0$ is the coefficient of magnetic resistivity. 

This model has interesting analogies with the vorticity formulation of the 3D Navier--Stokes and Euler equations, as well as with the 2D surface quasigeostrophic equations. Recall that the vorticity formulation of the Navier--Stokes equations in three dimensions is
\begin{equation}
\label{eqn:NSE-Vorticity}
\frac{\pd \bom}{\pd t} + (\uu \cdot \Grad) \bom - \eta \Laplace \bom = (\bom \cdot \Grad) \uu,
\end{equation}
where $\uu = K \ast \bom$ is given by the Biot--Savart law, with $K$ a homogeneous kernel of degree $-2$ (or rather of degree $1-n$ in dimension $n$). Our two-dimensional model has a very similar form --- compare \eqref{eqn:NSE-Vorticity} with \eqref{eqn:StokesMHD-B} --- but $\uu$ is instead given by
\[
\uu = K \ast (\BB \otimes \BB),
\]
where $K$ involves derivatives of the fundamental solution of the Stokes equation, and is homogeneous of degree $-1$. Unlike the 3D Navier--Stokes equations, for which existence and uniqueness of solutions for all time remains open, our two-dimensional model retains the essential features of the nonlinearities but admits a unique solution for all time.

Indeed, the main purpose of this paper is to prove the following theorem.

\begin{theorem}
\label{thm:MainTheorem}
Let $\Omega$ be one of the following:
\begin{itemize}
\item[$\bullet$] $\Omega \subset \R^{2}$ is a Lipschitz bounded domain with Dirichlet boundary conditions;
\item[$\bullet$] $\Omega = \R^{2}$; or
\item[$\bullet$] $\Omega = [0,1]^{2}$ with periodic boundary conditions.
\end{itemize}
Given an initial condition $\BB_{0} \in L^{2}(\Omega)$ with $\Div \BB_{0} = 0$, for any $T > 0$ there exists a unique weak solution $(\uu, \BB)$ of \eqref{eqn:StokesMHD}; that is, a pair of functions $(\uu, \BB)$ such that
\[
\uu \in L^{\infty}(0, T; L^{2,\infty}(\Omega)) \cap L^{2}(0, T; H^{1}(\Omega))
\]
and
\[
\BB \in L^{\infty}(0, T; L^{2}(\Omega)) \cap L^{2}(0, T; H^{1}(\Omega))
\]
satisfying \eqref{eqn:StokesMHD} as an equality in $L^{2}(0, T; H^{-1}(\Omega))$. Furthermore, for any $T > \eps > 0$ and any $k \in \N$,
\[
\uu, \BB \in L^{\infty}(\eps, T; H^{k}(\Omega)) \cap L^{2}(\eps, T; H^{k+1}(\Omega)).
\]
\end{theorem}

Our interest in system \eqref{eqn:StokesMHD} arises from its connection with the method of \emph{magnetic relaxation}, an idea due to Moffatt \cite{art:Moffatt1985}. He considers the related full MHD system:
\begin{subequations}
\label{eqn:MHD}
\begin{align}
\frac{\pd \uu}{\pd t} + (\uu \cdot \Grad) \uu - \nu \Laplace \uu + \Grad p_{*} &= (\BB \cdot \Grad) \BB, \\
\frac{\pd \BB}{\pd t} + (\uu \cdot \Grad) \BB - \eta \Laplace \BB &= (\BB \cdot \Grad) \uu, \\
\Div \uu = \Div \BB &= 0.
\end{align}
\end{subequations}
Formally, when $\eta = 0$, we obtain the standard energy estimate
\[
\frac{1}{2} \frac{\rd}{\rd t} \left( \norm{\uu}_{L^{2}}^{2} + \norm{\BB}_{L^{2}}^{2} \right) + \nu \norm{\Grad \uu}_{L^{2}}^{2} = 0;
\]
so as long as $\uu$ is not identically zero, the energy should decay. Furthermore, by using the so-called \emph{magnetic helicity}, which is preserved under the flow, we can find a lower bound for the energy of $\BB$: if $\mathscr{H}_{M} := \int_{\Omega} \BA \cdot \BB$, where $\Curl \BA = \BB$ is a vector potential for $\BB$, then
\[
C \norm{\BB}_{L^{2}}^{4} \geq \norm{\BB}_{L^{2}}^{2} \norm{\BA}_{L^{2}}^{2} \geq \left(\int_{\Omega} \BA \cdot \BB \right)^{2} = |\mathscr{H}_{M}|^{2} > 0.
\]
In other words, the magnetic forces on a viscous non-resistive plasma should come to equilibrium, so that the fluid velocity $\uu$ tends to zero. We are left with a steady magnetic field $\BB$ that satisfies $(\BB \cdot \Grad) \BB - \Grad p_{*} = 0$, which up to a change of sign for the pressure are the stationary Euler equations.

These arguments are heuristic, and as yet there is no rigorous proof that this method will yield a stationary Euler flow. The first problem is that it has not yet been proved that the system \eqref{eqn:MHD} with $\eta = 0$ has a unique solution for all time, even in two dimensions: the best known result is found in \cite{art:JiuNiu2006}, where short-time existence of strong solutions is proved by means of a ``vanishing resistivity'' argument, as well as a conditional regularity result; the conditional regularity result was later extended in \cite{art:FanOzawa2009}.

With $\eta > 0$, however, the existence theory for \eqref{eqn:MHD} is in a similar state to the Navier--Stokes equations, with global existence of weak solutions in two or three dimensions, and uniqueness in two dimensions; see \cite{art:DuvautLions1972} and \cite{art:SermangeTemam1983}. (Interestingly, global existence of weak solutions in two dimensions for the case $\nu = 0$ but $\eta > 0$ was proven in \cite{art:Kozono1989}, with various extensions in \cite{art:CST2003} and \cite{art:CT2002}, and conditional regularity results in \cite{art:FanOzawa2009} and \cite{art:ZhouFan2011}.)

The second problem is that, even with global existence and uniqueness, the system may not possess a limit state. Assuming that the equations have a smooth solution for all time, and furthermore that $\norm{\BB}_{\infty} \leq M$ for all time, N\'{u}\~{n}ez~\cite{art:Nunez2007} showed that (with $\eta = 0$) the kinetic energy must decay to zero, but that the magnetic field may not have a weak limit when a decaying forcing $f \in L^{2}(0, \infty; L^{2}(\Omega))$ is added to the $\uu$ equation.

If all we are interested in is the limiting state, the dynamical model used to obtain that steady state is not particularly important: in a talk given at the University of Warwick, Moffatt~\cite{misc:talkMoffatt2009} argued that dropping the acceleration terms from the $\uu$ equation and working with a ``Stokes'' model --- such as equations~\eqref{eqn:StokesMHD} --- might prove more mathematically amenable. As a first step towards a rigorous theory of magnetic relaxation for this model, this paper thus establishes existence and uniqueness theory for \eqref{eqn:StokesMHD} in the case $\eta > 0$ in two dimensions.

The proof of Theorem~\ref{thm:MainTheorem} is divided into several sections:
\begin{itemize}
\item In Section~\ref{sec:Interpolation}, we introduce the weak $L^{p}$ spaces, denoted $L^{p,\infty}$, and use the theory of interpolation spaces to prove the following generalised version of the 2D Ladyzhenskaya inequality:
\[
\norm{f}_{L^{4}} \leq c \norm{f}_{L^{2,\infty}}^{1/2} \norm{\Grad f}_{L^{2}}^{1/2}.
\]
\item In Section~\ref{sec:StokesElliptic} we consider elliptic regularity for the Stokes equations 
\begin{align*}
- \nu \Laplace \uu + \Grad p &= \Div \ff, \\
\Div \uu &= 0,
\end{align*}
and show that $\uu \in L^{2,\infty}$ whenever $\ff \in L^{1}$. 
\item In Section~\ref{sec:ExistenceUniqueness}, we use the results of the previous two sections to prove global existence and uniqueness of weak solutions for \eqref{eqn:StokesMHD} in a bounded domain $\Omega$ and the whole of $\R^{2}$.
\item In Section~\ref{sec:HigherOrder} we prove higher-order estimates to show that the solutions stay as smooth as the initial data permits for all time, and hence that after any arbitrary time $\eps > 0$ the solution is smooth.
\end{itemize}

\section{Interpolation and Ladyzhenskaya's inequality}
\label{sec:Interpolation}

In order to prove existence and uniqueness for our system \eqref{eqn:StokesMHD}, we will require a variant of Ladyzhenskaya's inequality. We first recall the standard inequality proved by Ladyzhenskaya~\cite{art:Ladyzhenskaya1958}: if $\Omega \subset \R^{2}$ is a bounded Lipschitz domain, then for $u \in H^{1}(\Omega)$,
\begin{equation}
\label{eqn:Ladyzhenskaya}
\norm{u}_{L^{4}} \leq c \norm{u}_{L^{2}}^{1/2} \norm{u}_{H^{1}}^{1/2}.
\end{equation}
One can prove this simply by using the embedding $H^{1/2} \subset L^{4}$ and interpolating $H^{1/2}$ between $L^{2}$ and $H^{1}$:
\[
\norm{u}_{L^{4}} \leq c \norm{u}_{H^{1/2}} \leq c \norm{u}_{L^{2}}^{1/2} \norm{u}_{H^{1}}^{1/2}.
\]
But one can also prove it directly for $u \in C^{1}_{c}(\Omega)$ (see \cite{art:Ladyzhenskaya1958}, or \cite{book:FMRT}, equation (4.8) on p.~17): we can write $u^{2} = 2 \int u \pd_{j} u \, \rd x_{j}$, and then integrate $u^{4} = (u^{2})^{2}$ to get
\begin{align*}
&\iint_{\R^{2}} |u^{4}| \, \rd x_{1} \, \rd x_{2} \\
&\qquad \qquad \leq \left( \int_{-\infty}^{\infty} \max_{x_{2}} |u(x)|^{2} \, \rd x_{1} \right) \left( \int_{-\infty}^{\infty} \max_{x_{1}} |u(x)|^{2} \, \rd x_{2} \right) \\
&\qquad \qquad \leq 4 \left( \iint_{\R^{2}} \max_{x_{2}} \abs{u \pd_{2} u} \, \rd x_{1} \rd x_{2} \right) \left( \iint_{\R^{2}} \max_{x_{1}} \abs{u \pd_{1} u} \, \rd x_{1} \rd x_{2} \right)
\end{align*}
and the result follows by applying the Cauchy--Schwarz inequality. The result for $u \in H^{1}(\Omega)$ follows by density of $C^{1}_{c}(\Omega)$ in $H^{1}(\Omega)$.

The variant of Ladyzhenskaya's inequality that we require is the standard inequality with $\norm{u}_{L^{2}}$ replaced with $\norm{u}_{L^{2, \infty}}$, where $L^{2, \infty}(\R^{2})$ is the \emph{weak $L^{2}$ space}:
\begin{equation}
\label{eqn:WeakLadyzhenskaya}
\norm{f}_{L^{4}} \leq c \norm{f}_{L^{2,\infty}}^{1/2} \norm{\Grad f}_{L^{2}}^{1/2}.
\end{equation}
In fact, in Section~\ref{sec:InterpSpaces} we will prove the stronger inequality
\begin{equation}
\label{eqn:WeakGN-BMO}
\norm{f}_{L^{p}} \leq c \norm{f}_{L^{q,\infty}}^{q/p} \norm{f}_{\BMO}^{1-q/p}
\end{equation}
for every $f \in L^{q,\infty}(\R^{n}) \cap \BMO(\R^{n})$, using the theory of interpolation spaces.

The inequality \eqref{eqn:WeakGN-BMO} is not altogether new: an alternative proof is sketched in \cite{art:KMW2007}, and it is a strengthening of the inequality
\begin{equation}
\label{eqn:GN-BMO}
\norm{f}_{L^{p}} \leq c \norm{f}_{L^{q}}^{q/p} \norm{f}_{\BMO}^{1-q/p},
\end{equation}
which has been proven a number of times before; see \cite{art:ChenZhu2005}, \cite{art:KozonoWadade2008}, \cite{art:DongXiao2011} and \cite{art:AzzamBedrossian}. In particular, the elegant proof of \eqref{eqn:GN-BMO} in Chen~\&~Zhu~\cite{art:ChenZhu2005}, which uses the John--Nirenberg inequality for functions in $\BMO$, is adapted in McCormick~et~al.~\cite{art:MJM} to give an alternative proof of \eqref{eqn:WeakGN-BMO}.

\subsection{Weak $L^{p}$ spaces}
\label{sec:WeakLp}

The weak $L^{p}$ spaces are defined as follows (see.~\cite{book:GrafakosClassical}, \S1.1, for example). Let $\Omega \subset \R^{n}$ be measurable. Given a measurable, a.e.-finite function $f \colon \Omega \to \R$, we define its \emph{distribution function} $d_{f} \colon [0, \infty) \to [0, \infty]$ by
\[
d_{f}(\alpha) := \mu \{ x \in \Omega : |f(x)| > \alpha \}.
\]
Then, given $1 \leq p < \infty$, the \emph{weak $L^{p}$ space}, denoted $L^{p,\infty}(\Omega)$, consists of all measurable, a.e.-finite functions $f$ for which the quantity
\[
\norm{f}_{p, \infty} := \inf \left\{ C > 0 : d_{f}(\alpha) \leq \frac{C^{p}}{\alpha^{p}} \text{ for all } \alpha > 0 \right\}
\]
is finite (see \cite{book:GrafakosClassical}, Definition~1.1.5). Note that $\norm{\cdot}_{p,\infty}$ is \emph{not} a norm, but only a quasi-norm --- the triangle inequality fails to hold, but instead we have the replacement inequality
\[
\norm{f+g}_{L^{p,\infty}} \leq 2 (\norm{f}_{L^{p,\infty}} + \norm{g}_{L^{p,\infty}}).
\]

It is fairly simple to see that $L^{p}(\Omega) \subsetneq L^{p,\infty}(\Omega)$: note that
\[
\norm{f}_{L^{p}}^{p} = \int_{\R^{n}} |f(x)|^{p} \geq \int\limits_{\{ x : |f(x)| > \alpha \} \hspace{-24pt}} |f(x)|^{p} \geq \alpha^{p} d_f(\alpha),
\]
so if $f \in L^{p}(\Omega)$ then $d_{f}(\alpha) \leq \norm{f}_{L^{p}}^{p} \alpha^{-p}$, and hence $f \in L^{p,\infty}(\Omega)$. However, the function $f(x) = |x|^{-n/p}$ is in $L^{p,\infty}(\R^{n})$, even though it is clearly not in $L^{p}(\R^{n})$: notice that
\[
d_{f}(\alpha) = \mu \{ x \in \R^{n} : |x|^{-n/p} > \alpha \} = \mu \{ x \in \R^{n} : |x| < \alpha^{-p/n} \} = \omega_{n} \alpha^{-p},
\]
where $\omega_{n}$ is the volume of the unit ball in $\R^{n}$, so $\norm{f}_{L^{p,\infty}} = \omega_{n}^{1/p}$.

Just as we can interpolate between $L^{p}$ spaces using H{\"o}lder's inequality, we can interpolate between weak $L^{p}$ spaces to get to a strong $L^{p}$ space in between. The following result is proved in \cite{book:GrafakosClassical}, Proposition~1.1.14.

\begin{lemma}[Weak--strong interpolation]
\label{lemma:WeakStrongInterp}
Let $1 < q < p < r  < \infty$. Then
\[
\norm{f}_{L^{p}} \leq c\norm{f}_{L^{q,\infty}}^{1-\alpha} \norm{f}_{L^{r,\infty}}^{\alpha},
\]
where $\frac{1 - \alpha}{q} + \frac{\alpha}{r} = \frac{1}{p}$.
\end{lemma}

We can also use weak $L^{p}$ spaces to weaken the standard Young's inequality for convolutions; see \cite{book:GrafakosClassical}, Theorem~1.2.13.

\begin{theorem}[Young's inequality]
\label{thm:Young}
Let $1 \leq p < \infty$ and $1 < q, r < \infty$ satisfy $\frac{1}{q} + 1 = \frac{1}{p} + \frac{1}{r}$. Then there exists a constant $C_{p,q,r} > 0$ such that, for all $f \in L^{p}$ and all $g \in L^{r, \infty}$, we have
\[
\norm{f \ast g}_{L^{q, \infty}} \leq c_{p,q,r} \norm{f}_{L^{p}} \norm{g}_{L^{r,\infty}}.
\]
\end{theorem}

\subsection{Interpolation spaces}
\label{sec:InterpSpaces}

In order to prove our weak version of Ladyzhenskaya's inequality, we will use some of the standard theory of \emph{interpolation spaces}. We recall here the basic facts we require: for full details, see the books of Bennett and Sharpley \cite{book:BennettSharpley}, \S5.1, and Bergh and L\"ofstr\"om \cite{book:BerghLofstrom}, \S3.1.

Let $(X_{0}, X_{1})$ be a compatible couple of Banach spaces (that is, there is a Hausdorff topological vector space $\mathfrak{X}$ such that $X_{0}$ and $X_{1}$ embed continuously into $\mathfrak{X}$). The \emph{$K$-functional} is defined for each $f \in X_{0} + X_{1}$ and $t > 0$ by
\[
K(f, t) = K(f, t; X_{0}, X_{1}) := \inf \{ \norm{f_{0}}_{X_{0}} + t \norm{f_{1}}_{X_{1}} : f = f_{0} + f_{1} \}
\]
where the infimum is taken over all representations $f = f_{0} + f_{1}$ of $f$ with $f_{0} \in X_{0}$ and $f_{1} \in X_{1}$.

For $0 \leq \theta \leq 1$, we define the interpolation space $(X_{0}, X_{1})_{\theta, \infty}$ as the space of all $f \in X_{0} + X_{1}$ for which the functional
\[
\norm{f}_{\theta, \infty} := \sup_{0 < t < \infty} t^{-\theta} K(f,t)
\]
is finite. A very useful property of interpolation spaces is the estimate on the norms:
\begin{equation}
\label{eqn:InterpolationEstimate}
\norm{f}_{\theta, \infty} \leq c \norm{f}_{X_{0}}^{1-\theta} \norm{f}_{X_{1}}^{\theta},
\end{equation}
when $0 \leq \theta \leq 1$ (see \cite{book:BerghLofstrom}, \S3.5, p.~49).

As a simple example of interpolation, note that
\[
(L^{1}(\R^{n}), L^{\infty}(\R^{n}))_{1 - 1/p, \infty} = L^{p,\infty}(\R^{n})
\]
if $p > 1$ (see \cite{book:BennettSharpley}, Chapter 5, Theorem~1.9). This is hardly surprising given that the definition of interpolation spaces is modelled on that of weak $L^{p}$ spaces. In fact, this equality remains true with $L^{\infty}$ replaced with $\BMO$ (see \cite{book:BennettSharpley}, Chapter 5, Theorem 8.11):
\[
(L^{1}(\R^{n}), \BMO(\R^{n}))_{1 - 1/p, \infty} = L^{p,\infty}(\R^{n}).
\]

The so-called \emph{reiteration theorem} allows us to interpolate between interpolation spaces: it says that when we interpolate between two interpolation spaces of the same couple $(X_{0}, X_{1})$, we get another interpolation space in the same family. 

\begin{theorem}[Reiteration Theorem]
\label{thm:Reiteration}
Let $(X_{0}, X_{1})$ be a compatible couple of Banach spaces, and let $0 \leq \theta_{0} < \theta_{1} \leq 1$. Set $A_{0} = (X_{0}, X_{1})_{\theta_{0}, \infty}$ and $A_{1} = (X_{0}, X_{1})_{\theta_{1}, \infty}$. If $0 < \theta < 1$, then
\[
(A_{0}, A_{1})_{\theta, \infty} = (X_{0}, X_{1})_{\theta', \infty}
\]
providing $\theta' = (1-\theta) \theta_{0} + \theta \theta_{1}$.
\end{theorem}

The proof may be found in \cite{book:BennettSharpley}, Chapter 5, Theorem 2.4, or \cite{book:BerghLofstrom}, Theorem 3.5.3. Using this, we can prove a weak version of our generalised Ladyzhenskaya inequality \eqref{eqn:WeakGN-BMO}.

\begin{lemma}[Weak interpolation]
\label{lem:WeakInterpolation}
For any $f \in L^{q, \infty}(\R^{n}) \cap \BMO(\R^{n})$, and any $q < p < \infty$,
\[
\norm{f}_{L^{p, \infty}} \leq c \norm{f}_{L^{q,\infty}}^{q/p} \norm{f}_{\BMO}^{1-q/p}.
\]
\end{lemma}

\begin{proof}
By \cite{book:BennettSharpley}, Chapter 5, Theorem 8.11, we have
\[
L^{q,\infty}(\R^{n}) = (L^{1}(\R^{n}), \BMO(\R^{n}))_{1 - 1/q, \infty}
\]
provided that $1 < q < \infty$. Set $\mathfrak{B} := (L^{1}(\R^{n}), \BMO(\R^{n}))_{1, \infty}$, and note that by \eqref{eqn:InterpolationEstimate} we have $\norm{f}_{\mathfrak{B}} \leq C \norm{f}_{\BMO}$. By the Reiteration Theorem (Theorem~\ref{thm:Reiteration}), we obtain
\[
L^{p,\infty}(\R^{n}) = (L^{q,\infty}(\R^{n}), \mathfrak{B})_{\alpha, \infty}
\]
with $q < p < \infty$, provided that $\alpha$ solves $1 - \frac{1}{p} = (1-\alpha) (1 - \frac{1}{q}) + \alpha \cdot 1$, or in other words that $\alpha = 1 - q/p$. Thus, using~\eqref{eqn:InterpolationEstimate}, we obtain
\[
\norm{f}_{L^{p,\infty}} \leq c \norm{f}_{L^{q,\infty}}^{q/p} \norm{f}_{\mathfrak{B}}^{1-q/p} \leq c \norm{f}_{L^{q,\infty}}^{q/p} \norm{f}_{\BMO}^{1-q/p},
\]
as required.
\end{proof}

By combining this with Lemma~\ref{lemma:WeakStrongInterp}, we obtain our generalised Ladyzhenskaya inequality \eqref{eqn:WeakGN-BMO}.

\begin{lemma}[Strong interpolation]
\label{lem:StrongInterpolation}
For any $f \in L^{q, \infty}(\R^{n}) \cap \BMO(\R^{n})$, and any $q < p < \infty$,
\[
\norm{f}_{L^{p}} \leq c \norm{f}_{L^{q,\infty}}^{q/p} \norm{f}_{\BMO}^{1-q/p}.
\]
\end{lemma}

\begin{proof}
Given $p > q$, choose any $r$ and $s$ such that $q < r < p < s < \infty$. Then
\[
\norm{f}_{L^{p}} \leq c \norm{f}_{L^{r,\infty}}^{1-\alpha} \norm{f}_{L^{s,\infty}}^{\alpha},
\]
where $\alpha$ satisfies $\frac{1 - \alpha}{r} + \frac{\alpha}{s} = \frac{1}{p}$. Applying Lemma~\ref{lem:WeakInterpolation} to each of the two factors on the right, we obtain
\begin{align*}
\norm{f}_{L^{p}} &\leq c (c \norm{f}_{L^{q,\infty}}^{q/r} \norm{f}_{\BMO}^{1-q/r})^{1-\alpha} (c \norm{f}_{L^{q,\infty}}^{q/s} \norm{f}_{\BMO}^{1-q/s})^{\alpha} \\
&\leq c \norm{f}_{L^{q,\infty}}^{(1-\alpha)q/r + \alpha q/s} \norm{f}_{\BMO}^{(1-\alpha)(1-q/r) + \alpha(1-q/s)} \\
&= c \norm{f}_{L^{q,\infty}}^{q/p} \norm{f}_{\BMO}^{1-q/p},
\end{align*}
as required.
\end{proof}

It follows from Poincar\'{e}'s inequality that $W^{1,n} \subset \BMO$ in $n \geq 2$ dimensions (see \S5.8.1 in \cite{book:Evans}): so, in particular, in two dimensions $\dot{H}^{1} \subset \BMO$. (One may also prove that $H^{n/2} \subset \BMO$ in $n$ dimensions: see Theorem~1.48 in \cite{book:BCD2011}.) Thus for $f \in L^{2,\infty}(\R^{2}) \cap \dot{H}^{1}(\R^{2})$, setting $n=2$, $p=4$ and $q=2$ in Lemma~\ref{lem:StrongInterpolation} we obtain \eqref{eqn:WeakLadyzhenskaya}:
\[
\norm{f}_{L^{4}} \leq c \norm{f}_{L^{2,\infty}}^{1/2} \norm{\Grad f}_{L^{2}}^{1/2}.
\]
When $\Omega$ is a bounded Lipschitz domain in $\R^{2}$, we may extend a function $f \in H^{1}_{0}(\Omega)$ by zero outside $\Omega$ and apply the above inequality on $\R^{2}$ to obtain the same for $\Omega$.

When $\Omega = [0,1]^{2}$ with periodic boundary conditions, however, a different argument using Fourier series is required: we obtain the same inequality using the Sobolev embedding $L^{4} \subset \dot{H}^{1/2}$ and the fact that
\begin{equation}
\label{eqn:Bernstein}
f = \sum_{|k| \leq \kappa} \hat{f}_k e^{2 \pi i k \cdot x} \qquad \implies \qquad \norm{f}_{L^{4}} \leq c \kappa^{1/2} \norm{f}_{L^{2,\infty}}
\end{equation}
(a weak form of Bernstein's inequality; see \cite{book:MuscaluSchlag1,book:MuscaluSchlag2} for the relevant theory of Fourier series and McCormick~et~al.~\cite{art:MJM} for the proof). Indeed, writing
\[
f = \sum_{|k| \leq \kappa} \hat{f}_k e^{2 \pi i k \cdot x} + \sum_{|k|>\kappa} \hat{f}_k e^{2 \pi i k \cdot x}
\]
we obtain, using \eqref{eqn:Bernstein} and $L^{4} \subset \dot{H}^{1/2}$,
\begin{align*}
\|f\|_{L^{4}} &\leq c \kappa^{1/2} \norm{f}_{L^{2,\infty}} + c \left(\sum_{|k|>\kappa} |k| |f_k|^2 \right)^{1/2}\\
&\leq c \kappa^{1/2} \norm{f}_{L^{2,\infty}} + c \kappa^{-1/2} \left( \sum_{|k|>\kappa} |k|^2 |f_k|^2 \right)^{1/2} \\
&\leq c \kappa^{1/2} \norm{f}_{L^{2,\infty}} + c \kappa^{-1/2} \norm{\Grad f}_{L^{2}}.
\end{align*}
Minimising over $\kappa$ we obtain \eqref{eqn:WeakLadyzhenskaya}. A similar argument involving Fourier transforms can be used to obtain a more general version of \eqref{eqn:WeakLadyzhenskaya} and \eqref{eqn:WeakGN-BMO} on the whole space; see McCormick~et~al.~\cite{art:MJM}.

\section{The Stokes operator and elliptic regularity in $L^{1}$}
\label{sec:StokesElliptic}

We now consider the Stokes equation alone. Take $\ff \colon \Omega \to \R^{2 \times 2}$, and define $\Div \ff \colon \Omega \to \R^{2}$ componentwise as follows:
\[
(\Div \ff)_{j} = \sum_{i=1}^{2} \pd_{i} \ff_{i,j}.
\]
With abuse of notation, we write $L^{p}(\Omega) = L^{p}(\Omega; \R^{2 \times 2})$ where no confusion can arise. We consider the equations
\begin{subequations}
\label{eqn:Stokes}
\begin{align}
- \nu \Laplace \uu + \Grad p &= \Div \ff, \\
\Div \uu &= 0,
\end{align}
\end{subequations}
with Dirichlet boundary conditions if $\Omega$ is bounded. By setting $\ff = \BB \otimes \BB$ (i.e.~$\ff_{i,j} = \BB_{i} \BB_{j}$) we recover equation~\eqref{eqn:StokesMHD-u}:
\[
(\Div (\BB \otimes \BB))_{j} = \sum_{i=1}^{2} \pd_{i} (\BB_{i} \BB_{j}) = \sum_{i=1}^{2} \BB_{i} \pd_{i} \BB_{j} = [(\BB \cdot \Grad) \BB]_{j},
\]
as $\BB$ is divergence-free. In this case, if $\BB \in L^{2}(\Omega)$, then $\BB \otimes \BB$ is in $L^{1}(\Omega)$, so the right-hand side behaves like the derivative of an $L^{1}$ function. If $\ff \in L^{p}(\Omega)$ for $p > 1$, one would expect that $\uu \in W^{1,p}(\Omega)$, but this does not hold for $p=1$. If it did, in two dimensions we would obtain $\uu \in W^{1,1}(\Omega) \subset L^{2}(\Omega)$. In fact, in this section we prove that, when $\ff \in L^{1}(\Omega)$ in \eqref{eqn:Stokes}, then $\uu \in L^{2,\infty}(\Omega)$.

The solution of this equation is given by convolution with the \emph{Green's function} or \emph{fundamental solution} of the equation: let $\UU$, $q$ solve
\begin{align*}
-\nu \Laplace \UU + \Grad q &= \delta, \\
\Div \UU &= 0,
\end{align*}
where $\delta$ denotes the Dirac delta function. Then the solution of \eqref{eqn:Stokes} is given by
\[
\uu = \UU \ast (\Div \ff), \qquad p = q \ast (\Div \ff).
\]
Integrating by parts, if $\pd_{k} \UU \in L^{2,\infty}(\Omega)$, then by Young's inequality (Theorem~\ref{thm:Young}) we can show that $\norm{\uu}_{L^{2,\infty}} \leq C \norm{\ff}_{L^{1}}$.

In the case $\Omega = \R^{2}$, we have explicit formulae for the fundamental solution $\UU$ and $q$:
\begin{align*}
\UU_{i,j}(x) &= \frac{1}{4\pi\nu} \left[ \frac{x_{i}x_{j}}{|x|^{2}} - \delta_{ij} \log |x| \right], \\
q_{j}(x) &= \frac{1}{2\pi} \frac{x_{j}}{|x|^{2}}.
\end{align*}
(A derivation of the fundamental solution may be found in \S{}IV.2 of \cite{book:Galdi}.) Integrating by parts with respect to $k$, we obtain
\begin{align*}
\uu_{i}(x) &= [\UU \ast (\Div \ff)]_{i}(x) \\
&= \sum_{j = 1}^{2} \int_{\R^{2}} \UU_{i,j} (x-y) \sum_{k=1}^{2} \pd_{k} f_{k,j}(y) \, \rd y \\
&= - \sum_{j, k = 1}^{2} \int_{\R^{2}}  \pd_{k} \UU_{i,j} (x-y) f_{k,j}(y) \, \rd y.
\end{align*}
Now,
\[
\pd_{k} \UU_{i,j} (x) = \frac{1}{4\pi\nu} \left[ \frac{\delta_{ik} x_{j} + \delta_{kj} x_{i}}{|x|^{2}} - \frac{x_{i} x_{j} x_{k}}{|x|^{4}} - \delta_{ij} \frac{x_{k}}{|x|^{2}} \right], 
\]
and so
\[
\abs{\pd_{k} \UU_{i,j} (x)} \leq \frac{1}{\pi\nu|x|}.
\]
As noted in Section~\ref{sec:WeakLp}, $\frac{1}{|x|}$ is in $L^{2,\infty}(\R^{2})$, and $\norm{\pd_{k} \UU_{i,j}}_{L^{2,\infty}} \leq \frac{1}{\nu \sqrt{\pi}}$. Using Young's inequality (Theorem~\ref{thm:Young}), we obtain
\begin{equation}
\label{eqn:StokesEstimate}
\norm{\uu}_{L^{2, \infty}} \leq c \norm{\pd_{k} \UU_{i,j}}_{L^{2, \infty}} \norm{\ff}_{L^{1}} \leq c \norm{\ff}_{L^{1}}.
\end{equation}
Thus, whenever $\ff \in L^{1}(\R^{2})$, $\uu \in L^{2,\infty}(\R^{2})$.

In the case where $\Omega = [0,1]^{2}$ with periodic boundary conditions, one can also write down an explicit formula for the fundamental solution --- see \cite{art:Hasimoto1959} and \cite{art:CichockiFelderhof1989}, for example --- and obtain \eqref{eqn:StokesEstimate} again; the details are very similar to the above case, and we omit them.

In the case where $\Omega$ is a bounded Lipschitz domain, while we no longer have an explicit formula for the Green's function $\UU$, by Theorem~7.1 in \cite{art:MitreaMitrea2011} we have $\Grad \UU \in L^{2,\infty}(\Omega)$ whenever $\Omega \subset \R^{2}$ is a bounded Lipschitz domain, thus using Young's inequality (Theorem~\ref{thm:Young}), we obtain \eqref{eqn:StokesEstimate} on a bounded Lipschitz domain as well; i.e., whenever $\ff \in L^{1}(\Omega)$, $\uu \in L^{2,\infty}(\Omega)$.

\section{Existence and uniqueness of weak solutions}
\label{sec:ExistenceUniqueness}

We return now to the system
\begin{subequations}
\label{eqn:StokesMHD-Rep}
\begin{align}
- \nu \Laplace \uu + \Grad p_{*} &= (\BB \cdot \Grad) \BB, \label{eqn:StokesMHD-Rep-u} \\
\frac{\pd \BB}{\pd t} + (\uu \cdot \Grad) \BB - \eta \Laplace \BB &= (\BB \cdot \Grad) \uu, \label{eqn:StokesMHD-Rep-B} \\
\Div \uu = \Div \BB &= 0,
\end{align}
\end{subequations}
where $p_{*} = p + \frac{1}{2} \abs{\BB}^{2}$. We will show that equations~\eqref{eqn:StokesMHD-Rep} have a unique weak solution for all time in the three cases of $\Omega$ described in Theorem~\ref{thm:MainTheorem}. We first define a weak solution.

\begin{definition}
A \emph{weak solution} of \eqref{eqn:StokesMHD-Rep} on $[0,T]$ is a pair of functions $(\uu, \BB)$ such that
\[
\uu \in L^{\infty}(0,T; L^{2,\infty}(\Omega)) \cap L^{2}(0,T; H^{1}(\Omega))
\]
and
\[
\BB \in L^{\infty}(0,T; L^{2}(\Omega)) \cap L^{2}(0,T; H^{1}(\Omega))
\]
satisfying \eqref{eqn:StokesMHD-Rep} as an equality in $L^{2}(0, T; H^{-1}(\Omega))$.
\end{definition}

Note that the pressure $p$ is uniquely determined by $\uu$ and $\BB$ by solution of a standard elliptic boundary value problem; see \cite{book:CDGG} or \cite{book:FMRT}. We will prove the following theorem.

\begin{theorem}
\label{thm:WeakSolution}
Given $\BB_{0} \in L^{2}(\Omega)$ with $\Div \BB_{0} = 0$, for any $T > 0$ there exists a unique weak solution $(\uu, \BB)$ of \eqref{eqn:StokesMHD-Rep}.
\end{theorem}

In Section~\ref{sec:BddDomain}, we will prove existence of a weak solution in the case $\Omega \subset \R^{2}$ is a Lipschitz bounded domain with Dirichlet boundary conditions, while in Section~\ref{sec:R2} we prove existence of a weak solution in the case $\Omega = \R^{2}$. The proof of existence in the case where $\Omega = [0,1]^{2}$ with periodic boundary conditions is analogous to the previous two, and we omit it. Finally, in Section~\ref{sec:Uniqueness}, we prove uniqueness of weak solutions.

\subsection{Global existence of solutions in a bounded domain}
\label{sec:BddDomain}

In this subsection we prove existence of a weak solution on a Lipschitz bounded domain $\Omega \subset \R^{2}$, with Dirichlet boundary conditions, using the method of Galerkin approximations. To do so, we first set up some notation. Let $H := \{ \uu \in L^{2}(\Omega) : \Div \uu = 0 \}$, and let $\Pi$ be the Leray projection $\Pi \colon L^{2}(\Omega) \to H$, i.e.~the orthogonal projection from $L^{2}$ onto $H$.

We define the \emph{Stokes operator} as $A := - \Pi \Laplace$.  Let $\{ \phi_{m} \}_{m \in \N} \subset C^{\infty}(\Omega)$ be the collection of eigenfunctions of the Stokes operator on $\Omega$ with Dirichlet boundary conditions, ordered such that the eigenvalues associated to $\phi_{m}$ are non-decreasing with respect to $m$. Let $V_{m}$ be the subspace of $H$ spanned by $\phi_{1}, \dots, \phi_{m}$, and let $P_{m} \colon H \to V_{m}$ be the orthogonal projection on to $V_{m}$.

In order to use the Galerkin method, we consider the equations
\begin{subequations}
\label{eqn:StokesMHD-Galerkin}
\begin{align}
- \nu \Laplace \uu^{m} + \Grad p^{m}_{*} &= (\BB^{m} \cdot \Grad) \BB^{m}, \label{eqn:StokesMHD-Galerkin-u} \\
\frac{\pd \BB^{m}}{\pd t} + P_{m} [(\uu^{m} \cdot \Grad) \BB^{m}] - \eta \Laplace \BB^{m} &= P_{m} [(\BB^{m} \cdot \Grad) \uu^{m}], \label{eqn:StokesMHD-Galerkin-B} \\
\Div \uu^{m} = \Div \BB^{m} &= 0.
\end{align}
\end{subequations}
Thinking of $\uu^{m}$ as a function of $\BB^{m}$ given by equation~\eqref{eqn:StokesMHD-Galerkin-u}, it is easy to check that \eqref{eqn:StokesMHD-Galerkin-B} is a locally Lipschitz ODE on the finite-dimensional space $V_{m}$, and thus by existence and uniqueness theory for finite-dimensional ODEs (Picard's theorem), there exists a unique solution $\BB^{m} \in V_{m}$ of equation~\eqref{eqn:StokesMHD-Galerkin-B}, with $\uu^{m}$ given by equation~\eqref{eqn:StokesMHD-Galerkin-u}.

\begin{proposition}[Energy estimates]
\label{prop:BddEnergyEstimates}
The Galerkin approximations are uniformly bounded in the following senses:
\begin{align*}
\uu^{m} &\in L^{\infty}(0, T; L^{2, \infty}(\Omega)) \cap L^{2}(0,T; H^{1}(\Omega));\\
\BB^{m} &\in L^{\infty}(0,T; L^{2}(\Omega)) \cap L^{2}(0,T; H^{1}(\Omega)).
\end{align*}
\end{proposition}

\begin{proof}
Take the inner product of equation~\eqref{eqn:StokesMHD-Galerkin-u} with $\uu^{m}$ and the inner product of equation~\eqref{eqn:StokesMHD-Galerkin-B} with $\BB^{m}$, and add to obtain
\[
\frac{1}{2} \frac{\rd}{\rd t} \norm{\BB^{m}(t)}_{L^{2}}^{2} + \nu \norm{\Grad \uu^{m}(t)}_{L^{2}}^{2} + \eta \norm{\Grad \BB^{m}(t)}_{L^{2}}^{2} = 0.
\]
Integrating over $[0,t]$ we obtain
\begin{align*}
&\norm{\BB^{m}(t)}_{L^{2}}^{2}  + 2\nu \int_{0}^{t} \norm{\Grad \uu^{m}(s)}_{L^{2}}^{2} \, \rd s + 2\eta \int_{0}^{t} \norm{\Grad \BB^{m}(s)}_{L^{2}}^{2} \, \rd s \\
&\qquad \qquad \qquad = \norm{\BB^{m}(0)}_{L^{2}}^{2} \leq \norm{\BB_{0}}_{L^{2}}^{2},
\end{align*}
and taking the essential supremum over all $t \in [0,T]$, it follows that
\begin{align*}
&\esssup_{t \in [0,T]} \norm{\BB^{m}(t)}_{L^{2}}^{2} + 2 \nu \int_{0}^{T} \norm{\Grad \uu^{m}(s)}_{L^{2}}^{2} \, \rd s + 2 \eta \int_{0}^{T} \norm{\Grad \BB^{m}(s)}_{L^{2}}^{2} \, \rd s \\
&\qquad \qquad \qquad \leq 2\norm{\BB_{0}}_{L^{2}}^{2},
\end{align*}
which proves that
\[
\uu^{m} \in L^{2}(0,T; H^{1}(\Omega)); \qquad \BB^{m} \in L^{\infty}(0,T; L^{2}(\Omega)) \cap L^{2}(0,T; H^{1}(\Omega)).
\]

As in Section~\ref{sec:StokesElliptic}, the solution $\uu^{m}$ to equation~\eqref{eqn:StokesMHD-Galerkin-u} is given by convolution with $\UU$, the Green's function for the Stokes equations. By \eqref{eqn:StokesEstimate}, we have
\[
\norm{\uu^{m}(t)}_{L^{2, \infty}} \leq c \norm{(\BB^{m}(t))^{2}}_{L^{1}} \leq c \norm{\BB^{m}(t)}_{L^{2}}^{2},
\]
so taking the essential supremum over $t \in (0, T)$ tells us that
\[
\uu^{m} \in L^{\infty}(0, T; L^{2, \infty}(\Omega)),
\]
which completes the proof.
\end{proof}

\begin{proposition}
\label{prop:Bdd-dBdtEstimate}
The Galerkin approximations are uniformly bounded as follows:
\[
\frac{\pd \BB^{m}}{\pd t} \in L^{2}(0, T; H^{-1}(\Omega)).
\]
\end{proposition}

\begin{proof}
Taking the $H^{-1}$ norm of the $\BB$ equation yields
\[
\bignorm{\frac{\pd \BB^{m}}{\pd t}}_{H^{-1}} \leq \eta \norm{\BB^{m}}_{H^{1}} + \norm{P_{m} [(\BB^{m} \cdot \Grad) \uu^{m}]}_{H^{-1}} + \norm{P_{m} [(\uu^{m} \cdot \Grad) \BB^{m}]}_{H^{-1}}.
\]
To estimate the nonlinear terms, we take the inner product with a test function $\phi \in H^{1}_{0}(\Omega)$ with $\norm{\phi}_{H^{1}} = 1$ and then take the supremum over such $\phi$:
\begin{align*}
\norm{P_{m} [(\BB^{m} \cdot \Grad) \uu^{m}]} &= \sup \abs{\inner{P_{m} [(\BB^{m} \cdot \Grad) \uu^{m}]}{\phi}} \\
&= \sup \abs{\inner{(\BB^{m} \cdot \Grad) P_{m}\phi}{\uu^{m}}} \\
&\leq \sup \norm{\BB^{m}}_{L^{4}} \norm{\uu^{m}}_{L^{4}} \norm{\Grad \phi}_{L^{2}} \\
&\leq \norm{\BB^{m}}_{L^{4}} \norm{\uu^{m}}_{L^{4}}.
\end{align*}
Similarly $\norm{P_{m} [(\uu^{m} \cdot \Grad) \BB^{m}]}_{H^{-1}} \leq \norm{\BB^{m}}_{L^{4}} \norm{\uu^{m}}_{L^{4}}$. By applying Ladyzhenskaya's inequality \eqref{eqn:Ladyzhenskaya} to $\BB^{m}$ and our weak Ladyzhenskaya's inequality \eqref{eqn:WeakLadyzhenskaya} to $\uu^{m}$, we obtain the the following estimate:
\[
\bignorm{\frac{\pd \BB^{m}}{\pd t}}_{H^{-1}}^{2} \leq \eta \norm{\BB^{m}}_{H^{1}}^{2} + c \norm{\BB^{m}}_{L^{2}} \norm{\BB^{m}}_{H^{1}} \norm{\uu^{m}}_{L^{2,\infty}} \norm{\uu^{m}}_{H^{1}},
\]
as required.
\end{proof}

So, in summary, we have the following:
\begin{align*}
\uu^{m} \text{ are uniformly bounded in } & L^{\infty}(0, T; L^{2,\infty}(\Omega)) \cap L^{2}(0, T; H^{1}(\Omega)), \\
\BB^{m} \text{ are uniformly bounded in } & L^{\infty}(0, T; L^{2}(\Omega)) \cap L^{2}(0, T; H^{1}(\Omega)), \\
\frac{\pd \BB^{m}}{\pd t}  \text{ are uniformly bounded in } & L^{2}(0, T; H^{-1}(\Omega)).
\end{align*}

To extract a convergent subsequence of $\BB^{m}$, we use the Aubin--Lions compactness lemma:

\begin{theorem}[Aubin--Lions compactness lemma]
\label{thm:AubinLions}
Let $X \subset B \subset Y$ be Banach spaces such that the inclusion $X \subset B$ is a compact embedding. Then, for any $1 < p < \infty$ and any $1 \leq q < \infty$, the space
\[
\left\{ f : f \in L^{p}(0, T; X) \text{ and } \frac{\pd f}{\pd t} \in L^{q}(0, T; Y) \right\}
\]
is compactly embedded in $L^{p}(0, T; B)$.
\end{theorem}

\begin{proof}
The original result of Aubin \cite{art:Aubin1963} and Lions \cite{book:Lions1969} covers the case when $1 < p, q < \infty$. Chapter 3 of Temam \cite{book:Temam} contains both the original case (see Theorem~2.1, p.~185), as well as the case $p=2$, $q=1$, whenever $X$ and $Y$ are Hilbert spaces (see Theorem~2.3, p.~187). The general case (and many other similar results) is proved in the paper of Simon \cite{art:Simon1987}, \S 8, Theorem 5 and Corollary 4.
\end{proof}

We use Theorem~\ref{thm:AubinLions} together with the Banach--Alaoglu compactness theorem to extract a subsequence, which we relabel as $\BB^{m}$, such that
\begin{align*}
\BB^{m} &\weakstarto \BB && \text{ in } L^{\infty}(0, T; L^{2}(\Omega)),\\
\BB^{m} &\weakto \BB && \text{ in } L^{2}(0, T; H^{1}(\Omega)),\\
\BB^{m} &\to \BB && \text{ in } L^{2}(0, T; L^{2}(\Omega)), \\
\frac{\pd \BB^{m}}{\pd t} &\weakstarto \frac{\pd \BB}{\pd t} && \text{ in } L^{2}(0, T; H^{-1}(\Omega)).
\end{align*}
Since the limit $\BB \in L^{\infty}(0, T; L^{2}(\Omega)) \cap L^{2}(0, T; H^{1}(\Omega))$, it is straightforward to show that $(\BB \cdot \Grad) \BB \in L^{2}(0, T; H^{-1}(\Omega))$. This allows us to \emph{define} $\uu$ to be the unique solution of
\begin{subequations}
\label{eqn:StokesRep}
\begin{align}
-\nu \Laplace \uu + \Grad p_{*} &= (\BB \cdot \Grad) \BB, \\
\Div \uu &= 0,
\end{align}
\end{subequations}
where $\uu \in L^{\infty}(0, T; L^{2,\infty}(\Omega)) \cap L^{2}(0, T; H^{1}(\Omega))$ by standard elliptic theory for the Stokes equations (see Section~\ref{sec:StokesElliptic} above, and Theorem~2.3 in Chapter~1 of \cite{book:Temam}). Having defined such a $\uu$, we now want to show that $\uu^{m}$ does indeed converge to this $\uu$ in the appropriate senses; this will allow us to show that the nonlinear terms involving $\uu$ converge and thus that the $\BB$ equation is satisfied in the limit.

\begin{proposition}
\label{prop:u-Convergence}
$\uu^{m} \to \uu$ strongly in $L^{2}(0, T; L^{2,\infty}(\Omega))$, and $\uu^{m} \weakto \uu$ weakly in $L^{2}(0, T; H^{1}(\Omega))$.
\end{proposition}

\begin{proof}
Subtracting the equations for $\uu^{m}$ and $\uu$, we obtain
\begin{align*}
-\nu \Laplace (\uu^{m} - \uu) + \Grad (p^{m}_{*} - p_{*}) &= \Div ( \BB^{m} \otimes \BB^{m} - \BB \otimes \BB) \\ &= \Div [\BB^{m} \otimes (\BB^{m} - \BB) + (\BB^{m} - \BB) \otimes \BB].
\end{align*}
By elliptic regularity, we obtain
\begin{align*}
\norm{\uu^{m} - \uu}_{L^{2,\infty}} &\leq c \norm{\BB^{m} \otimes (\BB^{m} - \BB)}_{L^{1}} + c \norm{(\BB^{m} - \BB) \otimes \BB}_{L^{1}} \\
&\leq c \norm{\BB^{m}}_{L^{2}} \norm{\BB^{m} - \BB}_{L^{2}} + c \norm{\BB^{m} - \BB}_{L^{2}} \norm{\BB}_{L^{2}} \\
&= c \norm{\BB^{m} - \BB}_{L^{2}} \left( \norm{\BB^{m}}_{L^{2}} + \norm{\BB}_{L^{2}} \right) \\
&\leq c(K+M) \norm{\BB^{m} - \BB}_{L^{2}},
\end{align*}
where $K = \sup_{m \in \N} \sup_{t \in [0,T]} \norm{\BB^{m}}_{L^{2}}$, $M = \sup_{t \in [0,T]} \norm{\BB}_{L^{2}}$. Squaring and integrating in time yields
\[
\int_{0}^{T} \norm{\uu^{m}(t) - \uu(t)}_{L^{2,\infty}}^{2} \, \rd t \leq c \int_{0}^{T} \norm{\BB^{m}(t) - \BB(t)}_{L^{2}}^{2} \, \rd t.
\]
As the right-hand side converges to zero, so is the left-hand side, and hence $\uu^{m} \to \uu$ strongly in $L^{2}(0, T; L^{2,\infty}(\Omega))$. Let $\vv$ be the weak limit of $\uu^{m}$ in $L^{2}(0, T; H^{1}(\Omega))$; \linebreak it remains to show that $\uu = \vv$. As $H^{1}(\Omega) \subset L^{2}(\Omega) \subset L^{2,\infty}(\Omega)$, we have \linebreak$(L^{2,\infty})^{*}(\Omega) \subset L^{2}(\Omega) \subset H^{-1}(\Omega)$. So if $\uu^{m} \weakto \vv$ in $L^{2}(0, T; H^{1}(\Omega))$, then $\uu^{m} \weakto \vv$ in $L^{2}(0, T; L^{2,\infty}(\Omega))$ (because we are testing with a smaller set of functionals). But $\uu^{m} \to \uu$ strongly (and hence also weakly) in $L^{2}(0, T; L^{2,\infty}(\Omega))$, and thus by uniqueness of weak limits $\uu = \vv$, and the proposition is proved.
\end{proof}

We now proceed to show that the nonlinear terms in the $\BB$ equation converge. The following proposition is symmetric in $\BB$ and $\uu$, and thus applies to both the $(\uu \cdot \Grad) \BB$ and $(\BB \cdot \Grad) \uu$ terms.

\begin{proposition}
\label{prop:NonlinearConvergence}
Suppose that:
\begin{itemize}
\item[$\bullet$] $\uu^{m} \to \uu$ and $\BB^{m} \to \BB$ (strongly) in $L^{2}(0, T; L^{2,\infty}(\Omega))$; and
\item[$\bullet$] $\uu^{m}$, $\BB^{m}$ are uniformly bounded in $L^{\infty}(0, T; L^{2,\infty}(\Omega)) \cap L^{2}(0, T; H^{1}(\Omega))$.
\end{itemize}
Then (after passing to a subsequence)
\[
P_{m}[(\uu^{m} \cdot \Grad) \BB^{m}] \weakstarto (\uu \cdot \Grad) \BB \quad \text{in } L^{2}(0, T; H^{-1}(\Omega)).
\]
\end{proposition}

\begin{proof}
We begin by showing that $P_{m}[(\uu^{m} \cdot \Grad) \BB^{m}]$ are uniformly bounded in \linebreak$L^{2}(0, T; H^{-1}(\Omega))$: observe that
\begin{align*}
\norm{P_{m}[(\uu^{m} \cdot \Grad) \BB^{m}]}_{H^{-1}} &= \sup \abs{\inner{(\uu^{m} \cdot \Grad) (P_{m} \phi)}{\BB^{m}}} \\
&\leq \norm{\uu^{m}}_{L^{4}} \norm{\BB^{m}}_{L^{4}} \sup \norm{\Grad \phi}_{L^{2}} \\
&\leq \norm{\uu^{m}}_{L^{4}} \norm{\BB^{m}}_{L^{4}} \\
&\leq c \norm{\uu^{m}}_{L^{2,\infty}}^{1/2} \norm{\Grad \uu^{m}}_{L^{2}}^{1/2} \norm{\BB^{m}}_{L^{2,\infty}}^{1/2} \norm{\Grad \BB^{m}}_{L^{2}}^{1/2}
\end{align*}
(where the supremum is taken over all $\phi \in H^{1}(\Omega)$ with $\norm{\phi}_{H^{1}} = 1$). Thus, squaring and integrating in time and applying H\"{o}lder's inequality shows that $P_{m}[(\uu^{m} \cdot \Grad) \BB^{m}]$ is uniformly bounded in $L^{2}(0, T; H^{-1}(\Omega))$, and hence that a subsequence of $P_{m}[(\uu^{m} \cdot \Grad) \BB^{m}]$ converges weakly-$*$ in $L^{2}(0, T; H^{-1}(\Omega))$; as usual we relabel this subsequence as the original sequence.

To show that the limit is indeed $(\uu \cdot \Grad) \BB$, we test with a slightly more regular test function. Let $\phi \in C^{0}(0, T; H^{1}(\Omega))$. Then
\begin{align*}
&\int_{0}^{T} \inner{P_{m}[(\uu^{m} \cdot \Grad) \BB^{m}] - (\uu \cdot \Grad) \BB}{\phi} \, \rd t \\
& = \underbrace{\int_{0}^{T} \inner{P_{m}[(\uu^{m} \cdot \Grad) \BB^{m} - (\uu \cdot \Grad) \BB]}{\phi} \, \rd t}_{\mathrm{I}} +  \underbrace{\int_{0}^{T} \inner{(\uu \cdot \Grad) \BB}{P_{m} \phi - \phi} \, \rd t}_{\mathrm{II}}.
\end{align*}
For the second integral, note that
\begin{align*}
\mathrm{II} \leq \norm{(\uu \cdot \Grad) \BB}_{H^{-1}} \norm{P_{m} \phi - \phi}_{H^{1}} \to 0
\end{align*}
as $m \to \infty$.
For the first integral, we have
\begin{align*}
\mathrm{I} &= \int_{0}^{T} \inner{(\uu^{m} \cdot \Grad) \BB^{m} - (\uu \cdot \Grad) \BB}{P_{m}\phi} \, \rd t  \\
&= \int_{0}^{T} \inner{(\uu^{m} \cdot \Grad) (\BB^{m} - \BB) + ((\uu^{m} - \uu) \cdot \Grad) \BB}{P_{m}\phi} \, \rd t \\
&\leq \int_{0}^{T} \left( \norm{\uu^{m}}_{L^{4}} \norm{\BB^{m} - \BB}_{L^{4}} + \norm{\uu^{m} - \uu}_{L^{4}} \norm{\BB}_{L^{4}} \right) \norm{\Grad \phi}_{L^{2}} \, \rd t \\
&\leq \max_{t \in [0,T]} \norm{\Grad \phi}_{L^{2}} \int_{0}^{T} \left( \norm{\uu^{m}}_{L^{4}} \norm{\BB^{m} - \BB}_{L^{4}} + \norm{\uu^{m} - \uu}_{L^{4}} \norm{\BB}_{L^{4}} \right)  \, \rd t.
\end{align*}
Now, by the weak Ladyzhenskaya inequality \eqref{eqn:WeakLadyzhenskaya},
\begin{align*}
&\int_{0}^{T} \norm{\uu^{m}}_{L^{4}} \norm{\BB^{m} - \BB}_{L^{4}} \, \rd t \\
&\qquad \leq \int_{0}^{T} \norm{\uu^{m}}_{L^{2,\infty}}^{1/2} \norm{\Grad \uu^{m}}_{L^{2}}^{1/2} \norm{\BB^{m} - \BB}_{L^{2,\infty}}^{1/2} \norm{\Grad(\BB^{m} - \BB)}_{L^{2}}^{1/2} \\
&\qquad \leq \left( \int_{0}^{T} \norm{\uu^{m}}_{L^{2,\infty}}^{2} \right)^{1/4} \cdot \left( \int_{0}^{T} \norm{\Grad \uu^{m}}_{L^{2}}^{2} \right)^{1/4} \cdot \\
&\qquad \qquad \qquad \cdot \left( \int_{0}^{T} \norm{\BB^{m} - \BB}_{L^{2,\infty}}^{2} \right)^{1/4} \cdot \left( \int_{0}^{T} \norm{\Grad(\BB^{m} - \BB)}_{L^{2}}^{2} \right)^{1/4}.
\end{align*}
The first, second and fourth integrals are bounded independent of $m$, and the third tends to zero as $m \to \infty$, so $\int_{0}^{T} \norm{\uu^{m}}_{L^{4}} \norm{\BB^{m} - \BB}_{L^{4}} \, \rd t \to 0$ as $m \to \infty$. Similarly, $\int_{0}^{T} \norm{\uu^{m} - \uu}_{L^{4}} \norm{\BB}_{L^{4}} \, \rd t \to 0$ as $m \to \infty$. Hence
\[
\int_{0}^{T} \inner{P_{m}[(\uu^{m} \cdot \Grad) \BB^{m}] - (\uu \cdot \Grad) \BB}{\phi} \, \rd t \to 0 \quad \text{for all } \phi \in C^{0}(0, T; H^{1}(\Omega)).
\]
Thus $P_{m}[(\uu^{m} \cdot \Grad) \BB^{m}] \weakstarto (\uu \cdot \Grad) \BB$ in $L^{2}(0, T; H^{-1}(\Omega))$ by uniqueness of weak-$*$ limits.
\end{proof}

Hence $(\uu, \BB)$ is indeed a weak solution of \eqref{eqn:StokesMHD-Rep}, which completes the proof of Theorem~\ref{thm:WeakSolution} in the case where $\Omega$ is a Lipschitz bounded domain in $\R^{2}$.

\subsection{Global existence of weak solutions in $\R^{2}$}
\label{sec:R2}

We turn now to the proof of Theorem~\ref{thm:WeakSolution} in the case $\Omega = \R^{2}$. Instead of Galerkin approximations, we mollify the equations, and then show convergence as $\eps \to 0$. The arguments, though, are not so different from those in the previous section, so we only outline the main changes.

Let $\rho \geq 0$ be a smooth function with compact support (i.e. $\rho \in C_{c}^{\infty}(\R^{n})$) such that $\int_{\R^{n}} \rho = 1$. Given $\eps > 0$, we define $\rho_{\eps}$ by $\rho_{\eps} (x) := \frac{1}{\eps^{n}} \rho (x/\eps)$, and we define the operator $\Jeps$ by $\Jeps v = \rho_{\eps} \ast v$, i.e.
\[
(\Jeps v)(x) = \int_{\R^{n}} \rho_{\eps}(x-y) v(y) \, \rd y.
\]
For brevity we will write $\tilde{v} := \Jeps v$. We consider the mollified MHD equations on the whole of $\R^{2}$ as follows:
\begin{subequations}
\label{eqn:StokesMHD-Mollified}
\begin{align}
- \nu \Laplace \uu^{\eps} + \Grad p_{*}^{\eps} &= \Jeps[(\tilde{\BB^{\eps}} \cdot \Grad) \tilde{\BB^{\eps}}], \label{eqn:StokesMHD-Mollified-u} \\
\frac{\pd \BB^{\eps}}{\pd t} + \Jeps[(\tilde{\uu^{\eps}} \cdot \Grad) \tilde{\BB^{\eps}}] - \eta \Jeps \Laplace \tilde{\BB^{\eps}} &= \Jeps[(\tilde{\BB^{\eps}} \cdot \Grad) \tilde{\uu^{\eps}}], \label{eqn:StokesMHD-Mollified-B} \\
\Div \uu^{\eps} = \Div \BB^{\eps} &= 0.
\end{align}
\end{subequations}
As in the previous section, we think of $\uu^{\eps}$ as a function of $\BB^{\eps}$. Then, using standard properties of the mollifier $\Jeps$ (see \cite{book:MajdaBertozzi}, Lemmas~3.5 and 3.6), it is straightforward to show that equation~\eqref{eqn:StokesMHD-Mollified-B} is a locally Lipschitz ODE on $H$ --- one basically follows the proof of Proposition~3.6 in \cite{book:MajdaBertozzi}. Hence by the Picard theorem for infinite-dimensional ODEs, equation~\eqref{eqn:StokesMHD-Mollified-B} will have a unique solution $\BB^{\eps} \in H$, so long as $\norm{\BB^{\eps}}_{L^{2}}$ remains finite, and $\uu^{\eps}$ is given by equation~\eqref{eqn:StokesMHD-Mollified-u}.

Repeating the estimates of Propositions~\ref{prop:BddEnergyEstimates} and \ref{prop:Bdd-dBdtEstimate}, with slight modifications to account for the extra mollifiers, we again have the following:
\begin{align*}
\uu^{\eps} \text{ are uniformly bounded in } & L^{\infty}(0, T; L^{2,\infty}(\R^{2})) \cap L^{2}(0, T; \dot{H}^{1}(\R^{2})), \\
\BB^{\eps} \text{ are uniformly bounded in } & L^{\infty}(0, T; L^{2}(\R^{2})),\\
\Jeps \BB^{\eps} \text{ are uniformly bounded in } & L^{2}(0, T; \dot{H}^{1}(\R^{2})), \\
\frac{\pd \BB^{\eps}}{\pd t}  \text{ are uniformly bounded in } & L^{2}(0, T; H^{-1}(\R^{2})).
\end{align*}
Because we are working on $\R^{2}$, we cannot apply the Aubin--Lions compactness lemma (Theorem~\ref{thm:AubinLions}) directly (because the embedding $H^{1} \subset L^{2}$ is no longer compact). Instead, there exists a subsequence of $\BB^{\eps}$ which converges strongly in $L^{2}(0, T; L^{2}(K))$ for any compact subset $K \subset \R^{2}$ (see Proposition~2.7 in \cite{book:CDGG}), and the limit satisfies 
\[
\BB \in L^{\infty}(0, T; L^{2}(\R^{2})) \cap L^{2}(0, T; H^{1}(\R^{2})).
\]

Thus, we may again define $\uu$ to be the unique solution of equation~\eqref{eqn:StokesRep}. A modification of Proposition~\ref{prop:u-Convergence} shows that a subsequence of $\uu^{\eps}$ converges strongly to $\uu$ in $L^{2}(0, T; L^{2}(K))$ for any compact subset $K \subset \R^{2}$, and this local strong convergence allows us to pass to the limit in the nonlinear terms: an argument similar to Proposition~\ref{prop:NonlinearConvergence} will show that (after passing to a subsequence)
\[
\Jeps[(\tilde{\uu^{\eps}} \cdot \Grad) \tilde{\BB^{\eps}}] \weakstarto (\uu \cdot \Grad) \BB, \qquad \Jeps[(\tilde{\BB^{\eps}} \cdot \Grad) \tilde{\uu^{\eps}}] \weakstarto (\BB \cdot \Grad) \uu
\]
in $L^{2}(0, T; H^{-1}(\R^{2}))$ (see \S2.2.4 of \cite{book:CDGG} for full details). Thus $(\uu, \BB)$ is indeed a weak solution of \eqref{eqn:StokesMHD-Rep}, which completes the proof of Theorem~\ref{thm:WeakSolution} in the case $\Omega = \R^{2}$.

\subsection{Uniqueness}
\label{sec:Uniqueness}

We now prove that weak solutions are unique. Note that the following proof applies equally in all three cases of Theorem~\ref{thm:MainTheorem}.

\begin{proposition}
Let $(\uu_{j}, \BB_{j})$, $j = 1, 2$, be two weak solutions with the same initial condition $\BB_{j}(0) = \BB_{0}$, such that 
\begin{align*}
\uu_{j} &\in L^{\infty}(0, T; L^{2, \infty}(\Omega)) \cap L^{2}(0, T; H^{1}(\Omega)), \\
\BB_{j} &\in L^{\infty}(0, T; L^{2}(\Omega)) \cap L^{2}(0, T; H^{1}(\Omega)).
\end{align*}
Then $\uu_{1} = \uu_{2}$ and $\BB_{1} = \BB_{2}$ as functions in the above spaces.
\end{proposition}

\begin{proof}
Take the equations for $(\uu_{1}, \BB_{1})$ and $(\uu_{2}, \BB_{2})$ and subtract: writing $\ww = \uu_{1} - \uu_{2}$, $\zz = \BB_{1} - \BB_{2}$ and $q = p_{1} - p_{2}$, we obtain
\begin{subequations}
\begin{align}
- \nu \Laplace \ww + \Grad q &= (\BB_{1} \cdot \Grad) \zz + (\zz \cdot \Grad) \BB_{2}, \label{eqn:MHDDiff-u} \\
\frac{\pd \zz}{\pd t} + (\uu_{1} \cdot \Grad) \zz + (\ww \cdot \Grad) \BB_{2} - \eta \Laplace \zz &= (\BB_{1} \cdot \Grad) \ww + (\zz \cdot \Grad) \uu_{2}, \label{eqn:MHDDiff-B} \\
\Div \ww = \Div \zz &= 0.
\end{align}
\end{subequations}
Taking the inner product of \eqref{eqn:MHDDiff-u} with $\ww$ yields
\[
\nu \norm{\Grad \ww}_{L^{2}}^{2} \leq (\norm{\BB_{1}}_{L^{4}} + \norm{\BB_{2}}_{L^{4}}) \norm{\zz}_{L^{4}} \norm{\Grad \ww}_{L^{2}},
\]
so dividing through and using Ladyzhenskaya's inequality we get
\begin{align}
\nu \norm{\Grad \ww}_{L^{2}} &\leq \norm{\zz}_{L^{4}} (\norm{\BB_{1}}_{L^{4}} + \norm{\BB_{2}}_{L^{4}}) \notag \\
&\leq  c \norm{\zz}_{L^{2}}^{1/2} \norm{\Grad \zz}_{L^{2}}^{1/2} (\norm{\BB_{1}}_{L^{2}}^{1/2} \norm{\Grad \BB_{1}}_{L^{2}}^{1/2} + \norm{\BB_{2}}_{L^{2}}^{1/2} \norm{\Grad \BB_{2}}_{L^{2}}^{1/2}) \notag \\
&\leq  c \norm{\zz}_{L^{2}}^{1/2} \norm{\Grad \zz}_{L^{2}}^{1/2} (\norm{\Grad \BB_{1}}_{L^{2}}^{1/2} + \norm{\Grad \BB_{2}}_{L^{2}}^{1/2}), \label{eqn:UniquenessH1Bound-w}
\end{align}
since $\BB_{j} \in L^{\infty}(0, T; L^{2}(\Omega))$. By the elliptic regularity arguments from Section~\ref{sec:StokesElliptic}, we obtain
\begin{equation}
\label{eqn:UniquenessL2WBound-w}
\norm{\ww}_{L^{2,\infty}} \leq c (\norm{\BB_{1} \zz}_{L^{1}} + \norm{\BB_{2} \zz}_{L^{1}}) \leq c \norm{\zz}_{L^{2}} (\norm{\BB_{1}}_{L^{2}} + \norm{\BB_{2}}_{L^{2}}) \leq c \norm{\zz}_{L^{2}}.
\end{equation}
Using the weak Ladyzhenskaya inequality \eqref{eqn:WeakLadyzhenskaya}, we obtain bounds in $L^{4}$ as follows:
\begin{equation}
\label{eqn:UniquenessL4Bound-w}
\norm{\ww}_{L^{4}} \leq \frac{c}{\nu} \norm{\zz}_{L^{2}}^{3/4} \norm{\Grad \zz}_{L^{2}}^{1/4} (\norm{\Grad \BB_{1}}_{L^{2}}^{1/4} + \norm{\Grad \BB_{2}}_{L^{2}}^{1/4}).
\end{equation}

By a similar argument, we can take the inner product with $\uu_{j}$ of the original equation for $\uu_{i}$ \eqref{eqn:StokesMHD-u} and obtain
\[
\norm{\Grad \uu_{j}}_{L^{2}} \leq c \norm{\BB_{j}}_{L^{4}}^{2} \leq c \norm{\BB_{j}}_{L^{2}} \norm{\Grad \BB_{j}}_{L^{2}},
\]
so using \eqref{eqn:WeakLadyzhenskaya} we get
\begin{equation}
\label{eqn:UniquenessL4Bound-u}
\norm{\uu_{j}}_{L^{4}} \leq c \norm{\uu_{j}}_{L^{2,\infty}}^{1/2} \norm{\Grad \uu_{j}}_{L^{2}}^{1/2} \leq c \norm{\BB_{j}}_{L^{2}}^{3/2} \norm{\Grad \BB_{j}}_{L^{2}}^{1/2} \leq c \norm{\Grad \BB_{j}}_{L^{2}}^{1/2},
\end{equation}
using elliptic regularity.

Since $\frac{\pd \zz}{\pd t} \in L^{2}(0, T; H^{-1}(\Omega))$ and $\zz \in L^{2}(0, T; H^{1}(\Omega))$, we can take the the inner product of \eqref{eqn:MHDDiff-B} with $\zz$ and obtain
\begin{align*}
&\frac{1}{2} \frac{\rd}{\rd t} \norm{\zz}_{L^{2}}^{2} + \eta \norm{\Grad \zz}_{L^{2}}^{2} \\
&\leq c \underbrace{\norm{\Grad \zz}_{L^{2}} \norm{\ww}_{L^{4}} (\norm{\BB_{1}}_{L^{4}} + \norm{\BB_{2}}_{L^{4}})}_{\mathrm{I}} + c \underbrace{\norm{\Grad \zz}_{L^{2}} \norm{\zz}_{L^{4}} ( \norm{\uu_{1}}_{L^{4}} + \norm{\uu_{2}}_{L^{4}} )}_{\mathrm{II}}.
\end{align*}
We bound $\mathrm{I}$ and $\mathrm{II}$ separately, using Ladyzhenskaya's inequality several times. By using \eqref{eqn:UniquenessL4Bound-w} and Young's inequality with $\eps = \frac{\eta}{4}$ for $8/5$ and $8/3$ on $\mathrm{I}$, we obtain
\begin{align*}
\mathrm{I} &\leq \frac{c}{\nu} \norm{\zz}_{L^{2}}^{3/4} \norm{\Grad \zz}_{L^{2}}^{5/4} (\norm{\Grad \BB_{1}}_{L^{2}}^{1/4} + \norm{\Grad \BB_{2}}_{L^{2}}^{1/4}) (\norm{\Grad \BB_{1}}_{L^{2}}^{1/2} + \norm{\Grad \BB_{2}}_{L^{2}}^{1/2}) \\
&\leq \frac{c}{\nu} \norm{\zz}_{L^{2}}^{3/4} \norm{\Grad \zz}_{L^{2}}^{5/4} (\norm{\Grad \BB_{1}}_{L^{2}}^{3/4} + \norm{\Grad \BB_{2}}_{L^{2}}^{3/4}) \\
&\leq \frac{\eta}{4} \norm{\Grad \zz}_{L^{2}}^{2} + \frac{c}{\nu \eta^{5/3}} \norm{\zz}_{L^{2}}^{2} (\norm{\Grad \BB_{1}}_{L^{2}}^{2} + \norm{\Grad \BB_{2}}_{L^{2}}^{2}).
\end{align*}
Similarly, by using \eqref{eqn:UniquenessL4Bound-u} and Young's inequality with $\eps = \frac{\eta}{4}$ for $4/3$ and $4$ on $\mathrm{II}$, we obtain
\begin{align*}
\mathrm{II} &\leq \norm{\zz}_{L^{2}}^{1/2} \norm{\Grad \zz}_{L^{2}}^{3/2}  (\norm{\Grad \BB_{1}}_{L^{2}}^{1/2} + \norm{\Grad \BB_{2}}_{L^{2}}^{1/2} ) \\
&\leq \frac{\eta}{4} \norm{\Grad \zz}_{L^{2}}^{2} + \frac{c}{\eta^{3}} \norm{\zz}_{L^{2}}^{2} (\norm{\Grad \BB_{1}}_{L^{2}}^{2} + \norm{\Grad \BB_{2}}_{L^{2}}^{2} ).
\end{align*}
Hence
\begin{equation}
\label{eqn:UniquenessFinal}
\frac{\rd}{\rd t} \norm{\zz}_{L^{2}}^{2} + \eta \norm{\Grad \zz}_{L^{2}}^{2} \leq \frac{c}{\nu \eta^{3}} \norm{\zz}_{L^{2}}^{2} (\norm{\Grad \BB_{1}}_{L^{2}}^{2} + \norm{\Grad \BB_{2}}_{L^{2}}^{2} ).
\end{equation}
By Gr\"{o}nwall's inequality, for all $t < T$ we have
\[
\norm{\zz(t)}_{L^{2}}^{2} \leq \norm{\zz_{0}}_{L^{2}}^{2} \exp \left( \int_{0}^{t} \norm{\Grad \BB_{1}(s)}_{L^{2}}^{2} \, \rd s + \int_{0}^{t} \norm{\Grad \BB_{2}(s)}_{L^{2}}^{2} \, \rd s \right).
\]
As $\BB_{j} \in L^{2}(0, T; H^{1}(\Omega))$, the integrals in the exponential are bounded, and since the initial conditions were the same, $\zz_{0} = 0$, and hence $\norm{\zz}_{L^{2}} = 0$ on $[0,T]$. By \eqref{eqn:UniquenessFinal}, $\norm{\Grad \zz}_{L^{2}} = 0$ on $[0,T]$. Hence, by \eqref{eqn:UniquenessL2WBound-w}, $\norm{\ww}_{L^{2,\infty}} = 0$ on $[0,T]$, and by \eqref{eqn:UniquenessH1Bound-w}, $\norm{\Grad \ww}_{L^{2}} = 0$ on $[0,T]$. This proves uniqueness of weak solutions.
\end{proof}

This completes the proof of Theorem~\ref{thm:WeakSolution}.

\section{Higher-order regularity estimates}
\label{sec:HigherOrder}

In this section, we prove the second part of Theorem~\ref{thm:MainTheorem}; that is, that the solution $(\uu, \BB)$ becomes smooth after an arbitrarily short time $\eps > 0$. In particular, we prove that if we start with initial data in $H^{k}(\Omega)$, then the solution stays in $H^{k}(\Omega)$ for all time:

\begin{theorem}
\label{thm:HigherOrder}
Let $k \in \N$. Suppose $\BB_{0} \in H^{k}(\Omega)$ with $\Div \BB_{0} = 0$. Then, for any $T > 0$, the unique weak solution of \eqref{eqn:StokesMHD-Rep} satisfies
\[
\uu, \BB \in L^{\infty}(0, T; H^{k}(\Omega)) \cap L^{2}(0, T; H^{k+1}(\Omega)).
\]
\end{theorem}

This immediately implies that the solution $(\uu, \BB)$ becomes smooth after an arbitrarily short time $\eps > 0$.

\begin{corollary}
Given any $T > \eps > 0$ and any $k \in \N$, the unique weak solution of \eqref{eqn:StokesMHD-Rep} satisfies
\[
\uu, \BB \in L^{\infty}(\eps, T; H^{k}(\Omega)).
\]
\end{corollary}

\begin{proof}
Fix $\eps > 0$. We already know that $\uu, \BB \in L^{2}(0, T; H^{1}(\Omega)$, so for some time $t_{1} < \eps/2$, $\uu(t_{1}), \BB(t_{1}) \in H^{1}(\Omega)$. Applying Theorem~\ref{thm:HigherOrder}, we obtain
\[
\uu, \BB \in L^{\infty}(\eps/2, T; H^{1}(\Omega)) \cap L^{2}(\eps/2, T; H^{2}(\Omega)). 
\]
Furthermore, if we know that
\[
\uu, \BB \in L^{\infty}(\eps(1 - 2^{1-k}), T; H^{k-1}(\Omega)) \cap L^{2}(\eps(1 - 2^{1-k}), T; H^{k}(\Omega)),
\]
then there is some time $t_{k}$ such that $\eps(1 - 2^{1-k}) < t_{k} < \eps(1 - 2^{-k})$ and $\uu(t_{k}), \BB(t_{k}) \in H^{k}(\Omega)$, and so applying Theorem~\ref{thm:HigherOrder}, we obtain
\[
\uu, \BB \in L^{\infty}(\eps(1 - 2^{-k}), T; H^{k}(\Omega)) \cap L^{2}(\eps(1 - 2^{-k}), T; H^{k+1}(\Omega)),
\]
The result follows by induction on $k$.
\end{proof}

We will prove Theorem~\ref{thm:HigherOrder} by induction on $k$. Since the base case and the induction step require rather different arguments, we split them into two separate propositions.

\begin{proposition}
Suppose $\BB_{0} \in H^{1}(\Omega)$ with $\Div \BB_{0} = 0$. Then, for any $T > 0$, the unique weak solution of \eqref{eqn:StokesMHD-Rep} satisfies
\[
\uu, \BB \in L^{\infty}(0, T; H^{1}(\Omega)) \cap L^{2}(0, T; H^{2}(\Omega)).
\]
\end{proposition}

\begin{proof}
Take the inner product of \eqref{eqn:StokesMHD-Rep-u} with $-\Laplace \uu$, the inner product of \eqref{eqn:StokesMHD-Rep-B} with $-\Laplace \BB$, and add:
\begin{align*}
&\frac{1}{2} \frac{\rd}{\rd t} \norm{\Grad \BB}_{L^{2}}^{2} + \nu \norm{\Laplace \uu}_{L^{2}}^{2} + \eta \norm{\Laplace \BB}_{L^{2}}^{2} \\
&\qquad \qquad = \inner{(\uu \cdot \Grad) \BB}{\Laplace \BB} - \inner{(\BB \cdot \Grad) \uu}{\Laplace \BB} - \inner{(\BB \cdot \Grad) \BB}{\Laplace \uu}.
\end{align*}
Now
\begin{align*}
\abs{\inner{(\uu \cdot \Grad) \BB}{\Laplace \BB}} &\leq \norm{\uu}_{L^{4}} \norm{\Grad \BB}_{L^{4}} \norm{\Laplace \BB}_{L^{2}} \\
&\leq c \norm{\uu}_{L^{2,\infty}}^{1/2} \norm{\Grad \uu}_{L^{2}}^{1/2} \norm{\Grad \BB}_{L^{2}}^{1/2} \norm{\Laplace \BB}_{L^{2}}^{3/2} \\
&\leq \frac{\eta}{6} \norm{\Laplace \BB}_{L^{2}}^{2} + c \norm{\uu}_{L^{2,\infty}}^{2} \norm{\Grad \uu}_{L^{2}}^{2} \norm{\Grad \BB}_{L^{2}}^{2}
\end{align*}
and
\begin{align*}
\abs{\inner{(\BB \cdot \Grad) \uu}{\Laplace \BB}} &\leq \norm{\BB}_{L^{4}} \norm{\Grad \uu}_{L^{4}} \norm{\Laplace \BB}_{L^{2}} \\
&\leq c \norm{\Grad \uu}_{L^{2}}^{1/2} \norm{\Laplace \uu}_{L^{2}}^{1/2} \norm{\BB}_{L^{2}}^{1/2} \norm{\Grad \BB}_{L^{2}}^{1/2} \norm{\Laplace \BB}_{L^{2}} \\
&\leq \frac{\eta}{6} \norm{\Laplace \BB}_{L^{2}}^{2} + \frac{\nu}{4} \norm{\Laplace \uu}_{L^{2}}^{2} + c \norm{\Grad \uu}_{L^{2}}^{2} \norm{\BB}_{L^{2}}^{2} \norm{\Grad \BB}_{L^{2}}^{2}
\end{align*}
and
\begin{align*}
\abs{\inner{(\BB \cdot \Grad) \BB}{\Laplace \uu}} &\leq \norm{\BB}_{L^{4}} \norm{\Grad \BB}_{L^{4}} \norm{\Laplace \uu}_{L^{2}} \\
&\leq c \norm{\BB}_{L^{2}}^{1/2} \norm{\Grad \BB}_{L^{2}} \norm{\Laplace \BB}_{L^{2}}^{1/2} \norm{\Laplace \uu}_{L^{2}} \\
&\leq \frac{\nu}{4} \norm{\Laplace \uu}_{L^{2}}^{2} + \frac{\eta}{6} \norm{\Laplace \BB}_{L^{2}}^{2} + c \norm{\BB}_{L^{2}}^{2} \norm{\Grad \BB}_{L^{2}}^{4}
\end{align*}
hence
\begin{align}
&\frac{1}{2} \frac{\rd}{\rd t} \norm{\Grad \BB}_{L^{2}}^{2} + \frac{\nu}{2} \norm{\Laplace \uu}_{L^{2}}^{2} + \frac{\eta}{2} \norm{\Laplace \BB}_{L^{2}}^{2} \notag \\
&\qquad \leq c \norm{\Grad \BB}_{L^{2}}^{2} \left( \norm{\uu}_{L^{2,\infty}}^{2} \norm{\Grad \uu}_{L^{2}}^{2} + \norm{\Grad \uu}_{L^{2}}^{2} \norm{\BB}_{L^{2}}^{2}  +  \norm{\BB}_{L^{2}}^{2} \norm{\Grad \BB}_{L^{2}}^{2} \right) . \label{eqn:H1-Estimate}
\end{align}
Since the integral of the last bracket is finite, by Gr\"onwall's inequality we get that $\BB \in L^{\infty}(0,T; H^{1}(\Omega))$, and hence reusing this bound in \eqref{eqn:H1-Estimate} yields $\uu, \BB \in L^{2}(0, T; H^{2}(\Omega))$. Finally, take the inner product of \eqref{eqn:StokesMHD-Rep-u} with $\uu$ to obtain
\[
\nu \norm{\Grad \uu}_{L^{2}}^{2} \leq \norm{\BB}_{L^{4}}^{2} \norm{\Grad \uu}_{L^{2}},
\]
so
\begin{align*}
\nu \norm{\Grad \uu}_{L^{2}} \leq \norm{\BB}_{L^{4}}^{2} \leq c \norm{\BB}_{L^{2}} \norm{\Grad \BB}_{L^{2}},
\end{align*}
and since the right-hand side is bounded, $\uu \in L^{\infty}(0, T; H^{1}(\Omega))$.
\end{proof}

To prove the induction step, we will need a higher-order estimate on the nonlinear term.

\begin{lemma}
\label{lemma:CF-Nonlinear}
Let $s > n/2$ be an integer, and let $\uu \in H^{s}(\Omega)$ and $\vv \in H^{s+1}(\Omega)$ such that $\Div \uu = \Div \vv = 0$. Then
\[
\norm{(\uu \cdot \Grad) \vv}_{H^{s}} \leq c_{\mathrm{N}} \norm{\uu}_{H^{s}} \norm{\vv}_{H^{s+1}}.
\]
\end{lemma}

\begin{proof}
When $\Omega$ is a bounded domain, this follows easily from the fact that $H^{s}$ is a Banach algebra for $s > n/2$ (see Theorem~4.39 in \cite{book:AdamsFournier}). For the periodic and $\R^{n}$ cases, this actually holds for any \emph{real} number $s > n/2$; see Lemma 10.4 in \cite{book:ConstantinFoias1988}.
\end{proof}

With this in hand, we proceed to the proof of the induction step.

\begin{proposition}
Fix an integer $k \geq 2$. Let $\BB_{0} \in H^{k}(\Omega)$ with $\Div \BB_{0} = 0$. Suppose that
\[
\uu, \BB \in L^{\infty}(0, T; H^{k-1}(\Omega)) \cap L^{2}(0, T; H^{k}(\Omega)).
\]
Then
\[
\uu, \BB \in L^{\infty}(0, T; H^{k}(\Omega)) \cap L^{2}(0, T; H^{k+1}(\Omega)).
\]
\end{proposition}

\begin{proof}
Take the inner product of \eqref{eqn:StokesMHD-Rep-u} with $(-1)^{k} \Laplace^{k} \uu$, the inner product of \eqref{eqn:StokesMHD-Rep-B} with $(-1)^{k} \Laplace^{k} \BB$, and add:
\begin{align*}
&\frac{1}{2} \frac{\rd}{\rd t} \norm{\BB}_{H^{k}}^{2} + \nu \norm{\uu}_{H^{k+1}}^{2} + \eta \norm{\BB}_{H^{k+1}}^{2} \\
&\; = (-1)^{k} \left[ \inner{(\BB \cdot \Grad) \uu}{\Laplace^{k} \BB} + \inner{(\BB \cdot \Grad) \BB}{\Laplace^{k} \uu} - \inner{(\uu \cdot \Grad) \BB}{\Laplace^{k} \BB} \right] \\
&\; \leq c \left[ \norm{(\BB \cdot \Grad) \uu}_{H^{k}} \norm{\BB}_{H^{k}} + \norm{(\BB \cdot \Grad) \BB}_{H^{k}} \norm{\uu}_{H^{k}} + \norm{(\uu \cdot \Grad) \BB}_{H^{k}} \norm{\BB}_{H^{k}} \right] \\
&\; \leq c \left[ \norm{\BB}_{H^{k}}^{2} \norm{\uu}_{H^{k+1}} + 2 \norm{\BB}_{H^{k}} \norm{\BB}_{H^{k+1}} \norm{\uu}_{H^{k}} \right] \\
&\; \leq \frac{\nu}{2} \norm{\uu}_{H^{k+1}}^{2} + \frac{\eta}{2} \norm{\BB}_{H^{k+1}}^{2} + c \left( \norm{\uu}_{H^{k}}^{2} + \norm{\BB}_{H^{k}}^{2} \right) \norm{\BB}_{H^{k}}^{2}
\end{align*}
so
\begin{equation}
\label{eqn:Hk-Estimate}
\frac{\rd}{\rd t} \norm{\BB}_{H^{k}}^{2} + \nu \norm{\uu}_{H^{k+1}}^{2} + \eta \norm{\BB}_{H^{k+1}}^{2} \leq c \norm{\BB}_{H^{k}}^{2} \left( \norm{\uu}_{H^{k}}^{2} + \norm{\BB}_{H^{k}}^{2} \right) .
\end{equation}
Since the integral of the last bracket is finite, by Gr\"onwall's inequality we get that $\BB \in L^{\infty}(0,T; H^{k}(\Omega))$, and hence reusing this bound in \eqref{eqn:Hk-Estimate} yields $\uu, \BB \in L^{2}(0, T; H^{k+1}(\Omega))$.

Finally, if $k \geq 3$ take the inner product of \eqref{eqn:StokesMHD-Rep-u} with $(-1)^{k-1} \Laplace^{k-1} \uu$ to obtain
\[
\nu \norm{\uu}_{H^{k}}^{2} \leq c \norm{(\BB \cdot \Grad) \BB}_{H^{k-1}} \norm{\uu}_{H^{k-1}} \leq \norm{\BB}_{H^{k-1}} \norm{\BB}_{H^{k}} \norm{\uu}_{H^{k-1}},
\]
and since the right-hand side is bounded, $\uu \in L^{\infty}(0,T; H^{k}(\Omega))$. In the case $k=2$, Lemma~\ref{lemma:CF-Nonlinear} does not apply to $\norm{(\BB \cdot \Grad) \BB}_{H^{k-1}}$, and so instead we take the inner product of \eqref{eqn:StokesMHD-Rep-u} with $- \Laplace \uu$ and estimate as follows:
\[
\nu \norm{\Laplace \uu}_{L^{2}}^{2} \leq \abs{\inner{(\BB \cdot \Grad) \BB}{\Laplace \uu}} \leq \norm{\BB}_{L^{4}} \norm{\Grad \BB}_{L^{4}} \norm{\Laplace \uu}_{L^{2}},
\]
so
\[
\norm{\Laplace \uu}_{L^{2}} \leq \norm{\BB}_{L^{2}}^{1/2} \norm{\Grad \BB}_{L^{2}} \norm{\Laplace \BB}_{L^{2}}^{1/2},
\]
and since the right-hand side is bounded, $\uu \in L^{\infty}(0,T; H^{2}(\Omega))$.
\end{proof}

This completes the proof of Theorem~\ref{thm:HigherOrder}, and hence also Theorem~\ref{thm:MainTheorem}.

\section{Non-resistive case ($\eta = 0$)}
\label{sec:NonResistive}

In the above we have developed an essentially complete theory of existence, uniqueness, and regularity for the system \eqref{eqn:StokesMHD} when $\eta>0$.

The non-resistive case ($\eta=0$) is much more difficult, and analogous to the vorticity formulation of the 3D Euler equations in the same way that the resistive system has similarlities with the 3D Navier-Stokes system (as discussed in the introduction). Two-dimensional models with similar structure to these canonical 3D equations (such as the 2D surface quasigeostrophic equation \cite{art:CMT1994}) have attracted considerable attention in recent years, and we plan to present an analysis of \eqref{eqn:StokesMHD} with $\eta=0$ in a future paper.

\bibliographystyle{abbrvnat}
\bibliography{MHDStokesPaper}   

\end{document}